\documentclass[12pt,twoside]{article}

\usepackage{t1enc}
\usepackage[english]{babel}
\usepackage[T1]{fontenc}
\usepackage[latin1]{inputenc}
\usepackage{amsmath}
\usepackage{amsfonts} 
\usepackage{mathrsfs} 
\usepackage{amsthm}
\usepackage{amssymb}
\usepackage[pdftex]{color,graphicx}
\usepackage{longtable}
\usepackage{verbatim}
\usepackage{indentfirst}
\usepackage{makeidx}
\usepackage{enumerate}
\usepackage{youngtab}
\usepackage{caption}
\usepackage{mathdots}
\usepackage{makeidx}
\usepackage{float}
\usepackage{amsbsy}

\usepackage[all]{xy}
\SelectTips{cm}{10}
\xyoption{web}

\DeclareMathOperator{\Ext}{Ext}
\DeclareMathOperator{\Def}{Def}
\DeclareMathOperator{\Inf}{Inf}
\DeclareMathOperator{\Indinf}{Indinf}
\DeclareMathOperator{\Defres}{Defres}

\newcommand{\IndinfDefres}{\Indinf^G_{A/B}\Iso_{\phi} \Defres_{C/D}^G}

\newcommand{\abs}[1]{\vert#1\vert} 
\newcommand{\gen}[1]{\ensuremath{\left\langle #1 \right\rangle}} 
\newcommand{\pg}[2]{\, ^{#1} \! #2 } %
\newcommand{\motdef}[1]{\textit{#1}}

\newcommand{\tq}{\: | \:}
\newcommand{\st}{\: | \:}

\newcommand*{\marge}[2][\false]{%
\ifx#1\false%
	\textit{#2}\marginpar{\footnotesize #2}
\else%
	\textit{#2}\marginpar{\footnotesize #1}
\fi}

{ \begin{list}%
	{$\bullet$}%
	{\setlength{\labelwidth}{30pt}%
	 \setlength{\leftmargin}{35pt}%
	 \setlength{\itemsep}{\parsep}}}%
{ \end{list} }

\newenvironment{indice}
{\begin{enumerate}[\textup{(}i\textup{)}] }
{\end{enumerate}} 

\newsavebox{\myQuoteSBOX}

\numberwithin{equation}{section}
\numberwithin{figure}{section}
\numberwithin{table}{section}

\DeclareMathOperator{\Aut}{Aut}
\DeclareMathOperator{\Out}{Out}

\DeclareMathOperator{\im}{Im} 
\DeclareMathOperator{\Iso}{Iso} 
\DeclareMathOperator{\Ker}{Ker} 
\DeclareMathOperator{\SL}{SL} 
\DeclareMathOperator{\Sp}{Sp} 

\DeclareMathOperator{\Ind}{Ind}
\DeclareMathOperator{\Res}{Res}

\newcommand{\cC}{\mathcal{C}}

\newcommand{\cP}{\mathcal{P}}

\newcommand{\IF}{\mathbb{F}}

\newcommand{\IZ}{\mathbb{Z}}

\newtheoremstyle{mythm}
	{10pt} 
	{3pt} 
	{\itshape}    
	{}    
	{\bfseries} 
	{.}   
	{ }   
	{\thmname{#1}\thmnumber{ #2}\thmnote{. #3}}    

\theoremstyle{mythm}
\newtheorem{thm}{Theorem}[section]

\newtheorem{lem}[thm]{Lemma}
\newtheorem{cor}[thm]{Corollary}

\newtheoremstyle{mydef}
	{10pt} 
	{3pt} 
	{}    
	{}    
	{\bfseries} 
	{.}   
	{ }   
	{\thmname{#1}\thmnumber{ #2}\thmnote{. #3}}    

\theoremstyle{mydef}

\newtheorem{rem}[thm]{Remark}

\begin{document}
\title{\textbf{Expansivity and Roquette Groups}}
\author{Alex Monnard \footnote{This work is part of a doctoral thesis prepared at the Ecole Polytechnique F\'ed\'erale de Lausanne, under the
supervision of Prof. Jacques Th\'evenaz.} }
\maketitle
\begin{abstract} 
One looks at  expansive subgroups in particular examples of Roquette groups. This study is motivated by the importance of expansive subgroups in the theory of stabilizing bisets highlighted in \cite{B-T}. In this paper we prove the non-existence of expansive subgroups with trivial $G$-core in different examples of Roquette groups $G$.
\end{abstract}
\textbf{Key words} : Roquette group, expansive subgroup, stabilizing biset.\\
\section{Introduction}
One goal in representation theory is to try to describe representations of a finite group from a subgroup or subquotient of order as small as possible. This has been studied in Green's theory of vertices and sources and Harish-Chandra induction for reductive groups (see for instance \cite{Harish} and \cite{Bouc96}). Another way to do so is to use stabilizing bisets introduced in \cite{B-T}. Indeed, let $k$ be a field, $G$ a finite group, $U$ a $(G,G)$-biset and $L$ a $kG$-module. Then $U$ is said to stabilize $L$ if $U(L) := kU \otimes_{kG} L $ is isomorphic to $L$. If we suppose that $L$ is indecomposable, then one can show that $U$ is of the form $\IndinfDefres:= \Ind_{A}^G \Inf_{A/B}^A \Iso_{\phi} \Def_{C/D}^C \Res_C^G$ for some subgroups $A,B,C,D$ and an isomorphism $\phi : C/D \rightarrow A/B $. In particular, this means that $L$ can be obtained from a representation of $A/B$. 

There is a close link between expansive subgroups and stabilizing bisets as exposed in \cite{B-T}. In particular, Theorem $7.3$ of \cite{B-T} states that if $k$ is a field, $G$ a finite group and $L$ a simple $kG$-module, then there exists an expansive subgroup $T$ of $G$ such that
$$\Indinf_{N_G(T)/T}^G\Defres_{N_G(T)/T}^G(L) \cong L.$$

This theorem proves the existence of stabilizing bisets for simple modules. It is possible that this biset is trivial, i.e. it is reduced to an isomorphism. The proof of the theorem shows that this could only be the case if $G$ is Roquette and $L$ is faithful.

This raises the question of proving the existence, or non-existence, of stabilizing bisets for Roquette groups. To do so, one way is to find expansive subgroups in Roquette groups and then look if it gives non-trivial stabilizing bisets. 

Note that expansive subgroups appear in the study of biset functors given  in \cite{Bouc}. They are used to define rational biset functors. This kind of biset functors has a wide variety of applications such as units of Burnside Rings, the Dade Group and the kernel of the linearization morphism. Although there is a classification of Roquette $p$-groups (see Chapter 5, Section 4 of \cite{Gor}), there are not many results on the existence of expansive or genetic subgroups. Therefore this article has interest beyond the scope of stabilizing bisets.

Let us end this introduction with a short description of the organization of the paper. In the first section one introduces the notions of Roquette groups and expansive subgroups. In particular, one presents the classes of Roquette groups studied in the rest of this paper, which are the first examples of Roquette groups. They give already difficult and interesting examples. Then the following sections are the study of these groups, with first Roquette $p$-groups in section 3, secondly groups with cyclic Fitting subgroup in section 4 and finally groups with extraspecial groups in the Fitting subgroup in section~5.

\section{Roquette groups and expansive subgroups}
We first recall the definition of Roquette groups and expansive subgroups and then one introduces the classes of examples treated in the rest of this paper.
\begin{indice}
\item A subgroup $T$ of a finite group $G$ is called \motdef{expansive} in $G$ if, for every $g \not\in N_G(T)$, the $N_G(T)$-core of the subgroup $\big(\pg{g}{T} \cap N_G(T)\big)T$ contains properly $T$. 
\item A finite group $G$ is said to be a \motdef{Roquette group} if all its normal abelian subgroups are cyclic.
\end{indice}
Let $G$ be a Roquette group and denote by $F(G)$ the Fitting subgroup of $G$, which is the product of the normal subgroups $O_p(G)$ for all primes $p$. As $G$ is Roquette each $O_p(G)$ does not contain a characteristic abelian subgroup which is not cyclic. By Theorem $4.9$ of \cite{Gor} on page $198$, such groups are known. More precisely, each subgroup $O_p(G)$ is the central product of an extraspecial group with a Roquette $p$-group. Roquette $p$-groups are known see Chapter 5, Section 4, in \cite{Gor}, so one starts our study with those groups. Then, one continues with groups with cyclic Fitting subgroup, corresponding to cyclic $O_p(G)$ for every prime $p$. Finally, one treats examples of groups with extraspecial groups in $O_p(G)$.

Motivated by the result of Theorem 9.3 of \cite{B-T}, the general idea is to prove that for the majority of these groups $G$ there is no non-trivial expansive subgroup with trivial $G$-core. To do so, for an arbitrary subgroup $T$ of $G$ with trivial $G$-core one finds a specific element $g$ in $G$ which is not in $N_G(T)$ such that $(\pg{g}{T} \cap N_G(T))$ is contained in $T$. Thus, the $N_G(T)$-core of the subgroup $(\pg{g}{T} \cap N_G(T))T $ is $T$ and so $T$ is not expansive. Remark also that it is equivalent to prove that $T$ is not expansive or $\pg{h}{T}$ is not expansive for a $h \in G$.

Finally, as mentioned in the introduction, our focus will be on stabilizing bisets for faithful modules and thus, by Proposition 1.5 of \cite{Mo}, on the study of expansive subgroups with trivial $G$-core.


\section{Roquette $p$-groups}\index{Roquette $p$-group}
In this section one looks at expansive subgroups in Roquette $p$-groups for $p$ a prime number. Let $P$ be such a group. One knows from Theorem 9.3 of \cite{B-T} that if $U$ is a stabilizing biset for a faithful simple $kP$-module, then $U$ has to be reduced to an isomorphism. Therefore there is no hope to use expansive subgroups to find a stabilizing biset. Nevertheless, as mentioned above we have an interest of understanding this notion. In this first case, an important ingredient is the classification of all Roquette $p$-groups, which we first recall.
\begin{lem} \label{lem91}
Let $p$ be a prime and let $P$ be a Roquette $p$-group of order $p^n$.
\begin{enumerate}
\item If $p$ is odd, then $P$ is cyclic.
\item If $p = 2$, then $P$ is cyclic, generalized quaternion (with $n \geq 3$), dihedral (with $n \geq 4$), or semi-dihedral (with $n \geq 4$).
\item If $P$ is cyclic or generalized quaternion, there is a unique subgroup $Z$ of order $p$. Any non-trivial subgroup contains $Z$.
\item If $P$ is dihedral and $Z = Z(P)$, then any non-trivial subgroup contains $Z$, except for two conjugacy classes of non-central subgroups of order $2$. If $T$ is a non-central subgroup of order $2$, then $S = N_P (T) = TZ$ is a Klein $4$-group and $N_P(S)$ is a (dihedral) group of order $8$.
\item If $P$ is semi-dihedral and $Z = Z(P)$, then any non-trivial subgroup contains $Z$, except for one conjugacy class of non-central subgroups of order $2$. If $T$ is a non-central subgroup of order $2$, then $S = N_P (T) = TZ$ is a Klein $4$-group and $N_P(S)$ is a (dihedral) group of order $8$.
\end{enumerate}
\end{lem}
\begin{proof}
See Chapter 5, Section 4, in \cite{Gor}
\end{proof}
Using this classification we are able to prove the non-existence of expansive subgroups with trivial $P$-core in a Roquette $p$-group.
\begin{thm}
Let $p$ be a prime number and let $P$ be a Roquette $p$-group. Then $P$ has no non-trivial expansive subgroup with trivial $P$-core.
\end{thm}
\begin{proof}
Let $T$ be a non-trivial subgroup with trivial $P$-core. Then $T \cap Z(P)$ has to be trivial, otherwise $T \cap Z(P)$ would be contained in the $P$-core of $T$. It follows from Lemma \ref{lem91} that $T$ is trivial, except possibly if $p=2$, $P$ is dihedral or semi-dihedral, and $T$ is a non-central subgroup of order $2$. To prove that such $T$ is not expansive one wants to look at $\pg{g}{T} \cap N_P(T)$ for a suitable element $g$ where $g \not\in N_P(T)$. 

In both cases $S=N_P(T)$ is a Klein group. Moreover, since $N_P (S)$ is (dihedral) of order $8$ and $P$ has order at least $16$, we can choose $g \not\in  N_P(S)$. Using such $g$ one has $\pg{g}{T} \cap N_P(T) = \pg{g}{T} \cap TZ = 1$ which proves that $T$ is not expansive as the $N_P(T)$-core of the subgroup $(\pg{g}{T} \cap N_P(T))T $ is exactly $T$ and so does not contain $T$ properly.
\end{proof}
\section{Groups with cyclic Fitting subgroup} \label{secexpcn}
In this section one wants to investigate groups $G$ such that the Fitting subgroup $F(G)$ is cyclic of order $n$. One wants to know if expansive subgroups with trivial $G$-core exist in such groups. 
In this section, one assumes that $G$ is solvable. First note that it is a well known fact that $C_G(F(G)) \leq F(G)$ and therefore $G/F(G)$ injects into $\Out(F(G))$ thus one has the following exact sequence
\begin{figure}[h]
$$\xymatrix@1{1 \ar[r] & C_n \ar[r] &G \ar[r] & S \ar[r] &1} $$
\end{figure}\\
where $S$ is a subgroup of $\Out(C_n) = \Aut(C_n)$. The map $\iota : C_n \rightarrow G$ is the inclusion map. The map $\pi : G \rightarrow S$ sends an element $g$ to the conjugation map $c_g$. Note that this is not enough to ensure that $G$ is Roquette at this stage. One discusses this issue later on, see Theorem \ref{suitex3cond}. 

Suppose $n = p_1^{k_1} \dots p_m^{k_m}$ for some primes $p_i$ and integers $k_i$, so $C_n = \prod_{i=1}^m C_{p_i^{k_i}}$. It's a well-known result that $\Aut(C_n) \cong \prod_{i=1}^m \Aut(C_{p_i^{k_i}})$. Recall also that $\Aut(C_{2^k}) \cong C_2 \times C_{2^{k-2}}$ and for an odd prime $p_i$ one has $\Aut(C_{p_i^{k_i}}) \cong C_{p_i-1} \times C_{p_i^{k_i-1}}$. Let $g_{p_i}$ be a generator of $C_{p_i^{k_i}}$ in $C_n$ and define $\alpha_{p_i}$ an element of $\Aut(C_n)$ by $$\alpha_{p_i} : g_{p_i} \mapsto g_{p_i}^{1+p_i^{k_i-1}}$$ and $\alpha_{p_i}(g_{p_j}) = g_{p_j}$ if $j \neq i$. The map $\alpha_{p_i}$ is an element of order $p_i$ if $k_i > 1$, otherwise it is the identity map.
\begin{lem} \label{h1trivial}
Let $p$ be a prime number dividing $n$, then 
$$H^1\big(\gen{\alpha_p},C_n\big)  = \left\{
    \begin{array}{ll}
        1 & \text{ if $p$ is odd or $p=2$ and $k>2$,} \\
        C_2 & \text{ if $p=2$ and $k=2$.}
    \end{array}
\right.$$ 
\end{lem}
\begin{proof}
Decompose $n$ as $n=p^k  \cdot n/p^k$ such that $p$ does not divide $n/p^k$ and let $g$ be a generator of $C_{p^k}$ in $C_n$. 
Note that 
$$H^1\big(\gen{\alpha_p},C_n\big) \cong H^1\big(\gen{\alpha_p},C_{p^k}\big) \times H^1\big(\gen{\alpha_p},C_{n/p^k}\big)$$ but  $H^1\big(\gen{\alpha_p},C_{n/p^k}\big)$ is trivial because the order of $\gen{\alpha_p}$ and the order of $C_{n/p^k}$ are coprime. Therefore $H^1\big(\gen{\alpha_p},C_n\big)$ is equal to $H^1\big(\gen{\alpha_p},C_{p^k}\big)$. Moreover, recall that in the cyclic case $H^1\big(\gen{\alpha_p},C_{p^k}\big) = \Ker(t) / \im(\alpha_p \cdot \upsilon )$, where $t = \prod_{i=0}^{p-1} \alpha_p^i$, $\upsilon \in \Aut(C_{p^k})$ sends $g$ to $g^{-1}$ and $(\alpha_p \cdot \upsilon)(g) := \alpha_p(g) \upsilon (g)$. Let's start to describe the action of $t$ on $C_{p^k}$. 
\begin{eqnarray*}
t(g^j) &=& \prod_{i=0}^{p-1} \alpha_p^i(g)^j = \prod_{i=0}^{p-1} g^{j(1+p^{k-1})^i} = \prod_{i=0}^{p-1}  g^{j+jip^{k-1} + \dots} \\
&=& \prod_{i=0}^{p-1}  g^{j+jip^{k-1}} = g^{pj+jp^{k-1} \frac{p(p-1)}{2} }.
\end{eqnarray*}
The equality between the first and the second line holds because the power of $p$ for the rest of the terms is bigger than $k$. If $p$ is odd the last term is equal to $g^{pj}$. Therefore $t$ sends $g$ to $g^p$, its kernel is $\gen{g^{p^{k-1}}}$. But the image of $\alpha_p \cdot \upsilon$ is also $\gen{g^{p^{k-1}}}$ as $(\alpha_p \cdot \upsilon)(g) =g^{1+p^{k-1}} g^{-1} = g^{p^{k-1}}$.
If $p=2$, the map $t$ sends $g^j$ to $g^{j(p+p^{k-1})}$. The kernel is again $\gen{g^{p^{k-1}}}$ if $k >2$ and $\gen{g^{p^{k-2}}}$ if $k=2$. 
As the kernel is a subgroup of the cyclic group $C_{2^k}$ it's easy to check it by hand. The image of $\alpha_p \cdot \upsilon$ is also $\gen{g^{p^{k-1}}}$ as $(\alpha_p \cdot \upsilon)(g) =g^{1+p^{k-1}} g^{-1} = g^{p^{k-1}}$. This leads us to the conclusion.
\end{proof}
\begin{lem} \label{h2trivial}
Let $p$ be a prime number dividing $n$, then 
$$H^2\big(\gen{\alpha_p},C_n\big)  = \left\{
    \begin{array}{ll}
        1 & \text{ if $p$ is odd or $p=2$ and $k>2$,} \\
        C_2 & \text{ if $p=2$ and $k=2$.}
    \end{array}
\right.$$ 
\end{lem}
\begin{proof}
Again, decompose $n$ as $n=p^k  \cdot n/p^k$ such that $p$ does not divide $n/p^k$ and let $g$ be a generator of $C_{p^k}$ in $C_n$. Using the same argument, one has $H^2\big(\gen{\alpha_p},C_n\big) = H^2\big(\gen{\alpha_p},C_{p^k}\big)$. In the cyclic case $H^2\big(\gen{\alpha_p},C_{p^k}\big) = C_{C_{p^k}} (\gen{ \alpha_p}) / \im(t)$, where $t = \prod_i \alpha_p^i$. First, note that $C_{C_{p^k}} (\gen{ \alpha_p}) = \gen{g^p}$. Indeed, it is easy to check that $C_{C_{p^k}} (\gen{ \alpha_p}) \geq \gen{g^p}$ but  $\gen{g^p}$ is a maximal subgroup of  $\gen{g}$ and $g$ is not stabilized by $\alpha_p$ therefore the other inclusion follows. Secondly, using the description of the action of $t$ in Lemma \ref{h1trivial}, one has also that $\im(t) = \gen{g^p}$ if $p$ is odd or $p=2$ and $k>2$ and therefore $H^2\big(\gen{\alpha_p},C_n\big)$ is trivial for these cases. Nevertheless if $p=2$ and $k=2$ then $\im(t) =1$ and so $H^2\big(\gen{\alpha_p},C_n\big) = C_2.$
\end{proof}
\begin{cor} \label{corexistenceD}
Let $p$ be an odd prime or $p = 2$ and $k > 2$. Suppose $\gen{\alpha_p}$ is a subgroup of $S$. Then there exists a subgroup $D$ of $G$ such that $\pi(D) = \gen{\alpha_p}$ and $D\cap C_n =1$.
\end{cor}
\begin{proof}
By Lemma \ref{h2trivial}, there exists only one class of extensions of $\gen{\alpha_p}$ by $C_n$. Therefore the extension $\pi^{-1}(\gen{\alpha_p})$ is the semi-direct product of $C_n$ by a cyclic group of order $p$, which is the subgroup $D$ that we are looking for.
\end{proof}
\begin{lem} \label{lemngcd}
Let $G$ be an extension of $S$ by $C_n$ as above. Let $D$ be a subgroup of $G$ such that $D \cap C_n = 1$, then $N_{C_n}(D) = C_{C_n} (D) = C_{C_n}(\pi(D))$.
\end{lem}
\begin{proof}
For the first equality, let $x$ be an element of $N_{C_n}(D)$. Then, for all $d \in D$ one has $x d x^{-1} \in D.$ But $x d x^{-1} = x \pg{d}{x^{-1}} d$ which belongs to $D$ if, and only if,  $x \pg{d}{x^{-1}}  = 1$ which means that $x =\pg{d}{x}$ . This implies that $x$ is an element of $C_{C_n}(D)$. The other inclusion is trivial.

For the second equality, note that the action of $D$ on $C_n$ is the same as the action of $\pi(D)$ on $C_n$ by definition of the map $\pi$.
\end{proof}
\begin{lem} \label{cproj}
Let $H$ be a subgroup of $S$ and $H_i$ the $i^{\text{th}}$-projection of $H$ on $\Aut(C_n)$. Then 
$$C_{C_n} (H) =\prod_{i=1}^m C_{C_{p_i^{k_i}}} (H_i).$$
\end{lem}
\begin{proof}
Recall that $C_n = \prod_{i=1}^m C_{p_i^{k_i}}$. Now this is just a calculation :
\begin{eqnarray*}
C_{C_n} (H)  &=& \{ c = (c_1,\dots, c_m) \in C_n  \tq \pg{h}{c} = c \text{ for all } h \in H \} \\
&=& \{ c = (c_1,\dots, c_m) \in C_n  \tq \pg{h}{c_i} = c_i \text{ for all } i \text{ and for all }h \in H \} \\
&=& \prod_{i=1}^m \{ c_i \in C_{p_i^{k_i}}  \tq \pg{h}{c_i} = c_i \text{ for all } i \text{ and for all }h \in H \} \\
&=& \prod_{i=1}^m \{ c_i \in C_{p_i^{k_i}}  \tq \pg{h_i}{c_i} = c_i \text{ for all } i \text{ and for all }h_i \in H_i \} \\
&=& \prod_{i=1}^m C_{C_{p_i^{k_i}}} (H_i).
\end{eqnarray*}
\end{proof}
\begin{lem} \label{p2}
Let $G$ be the group $C_{2^k} \rtimes C_2$ with $k >2$, where $C_2$ is generated by either $\beta_1 : g \mapsto g^{-1}$ or $\beta_2 : g \mapsto g^{-1+2^{k-1}}$ where $g$ is a generator of $C_{2^k}$. Let $b$ be an element of $C_2$. If the element $g \pg{b}{g^{-1}}$ belongs to $C_{C_{2^k}}(C_2)$ then the only possibility is that $b = 1$.
\end{lem}
\begin{proof}
Note that in both cases $C_{C_{2^k}}(C_2)=  \{ c \in C_{2^k} \tq c^2 = 1\} $. Suppose now that $g \pg{b}{g^{-1}}$ belongs to $C_{C_{2^k}}(C_2)$, where $b$ is an element of $C_2$. So actually, except being the identity, $b$ could only be $\beta_1$ or $\beta_2$ depending in which case we are. One shows that it must imply $b=1$ anyway. Indeed, remark that $g\beta_1(g^{-1}) = g^2$ is not an element of $C_{C_{2^k}}(C_2)$ because $k > 2$. Similary for $g\beta_2(g^{-1}) = g^{2+2^{k-1}}$. Therefore in both cases one must have $b=1$.
\end{proof}
\begin{thm} \label{suitex3cond}
Let $G$ be a group such that there is an exact sequence
\begin{figure}[H]
$$\xymatrix@1{1 \ar[r] & C_n \ar[r] &G \ar[r] & S \ar[r] &1} $$
\end{figure}
where $S$ is a subgroup of $\Aut(C_n)$, the map $\iota : C_n \rightarrow G$ is the inclusion map, the map $\pi : G \rightarrow S$ sends an element $g$ to the conjugation map $c_g$ and the power of the prime $2$ in the decomposition of $n$ is different from $2$. Then 
 \begin{enumerate}
\item if $G$ is a Roquette group then $S$ does not contain any subgroup $\gen{ \alpha_p}$ for a prime $p$ dividing $n$.
\item If $S_{p_j}$ does not contain a subgroup $\gen{ \alpha_{p_j}}$ for all prime $p_j$ dividing $n$ then $G$ has no non-trivial expansive subgroup with trivial $G$-core, where $S_{p_j}$ denotes the $j^{\text{th}}$-projection of $S$ on $\Aut(C_n)$.
\item If $G$ has no non-trivial expansive subgroup with trivial $G$-core then $G$ is Roquette.
\end{enumerate}
\end{thm}
\begin{proof}
One proves $1$ by proving the converse. Suppose there exists a prime $p$ such that $\gen{ \alpha_p}$ is a subgroup of $S$. Decompose $n$ as $n=p^{k}  \cdot n/p^{k} $ such that $p$ does not divide $n/p^{k} $ and let $g_p$ be a generator of $C_{p^{k} }$ in $C_n$.
By Corollary \ref{corexistenceD}, let $D$ be a subgroup of $G$ such that $\pi(D) = \gen{ \alpha_p}$ and $D \cap C_n = 1$. By Lemma \ref{lemngcd} and a quick calculation, one has $N_{C_n}(D) = C_{C_n} (\gen{ \alpha_p}) = \gen{(g_p)^p} \times C_{n/p^k}$. Indeed, it is easy to check that $C_{C_n} (\gen{ \alpha_p}) \geq \gen{(g_p)^p} \times C_{n/p^k}$ but  $\gen{(g_p)^p}$ is a maximal subgroup of  $\gen{(g_p)}$ and $g_p$ is not stabilized by $\alpha_p$ therefore the other inclusion follows.  Define the subgroup $E := \gen{(g_p)^{p^{k-1}} } \times D$, which is an elementary abelian $p$-group. One shows that $E$ is a normal subgroup of $G$. Indeed, let $h$ be an element of $G$, then $\pg{h}{E} = \pg{h}{\gen{(g_p)^{p^{k-1}} }} \times \pg{h}{D} = \gen{(g_p)^{p^{k-1}} } \times  \pg{h}{D}$ as $C_n$ is a normal subgroup of $G$. Moreover, by Lemma \ref{h1trivial}, one has $\pg{h}{D} = \pg{c}{D}$ for some $c \in C_{p^k}$, because $\pg{h}{D}$ is a complement of $C_n$ in $C_n \rtimes D$ and the elements of $D$ act trivially on $C_{n/p^k}$. Indeed one has $\pg{h}{D} \cap C_n = \pg{h}{(D \cap \pg{h^{-1}}{C_n})} = \pg{h}{(D \cap C_n)} =1$ and $\pg{h}{D} \subset \pi^{-1} \pi (\pg{h}{D}) = \pi^{-1}(\pg{\pi(h)} {\pi(D)}) =\pi^{-1} (\pi(D)) = C_n \rtimes D$. As $\alpha_p$ acts trivially on $C_{n/p^k}$ so does $D$ and thus one can restrict the conjugation to an element $c \in C_{p^k}$ instead of $C_n$. Now one looks at $\pg{c}{D}$. Let $x$ be an element of $\pg{c}{D}$, then $x = c \pg{d}{c^{-1}} d$ for a $d$ in $D$. Write $c$ as $g_p^i$ and $\pi(d)$ as $ \alpha_p^j$. Recall that the action of $D$ on $C_n$ is the same as $\gen{ \alpha_p}$. Therefore $c\pg{d}{c^{-1}} = c\pg{\alpha_p^j}{c^{-1}} = (g_p)^{-ijp^{k-1}}$ which implies that $\pg{c}{D}$ is included in $\gen{(g_p)^{p^{k-1}} } \times D$ and so $E$ is normal.  Therefore $G$ is not Roquette. This proves\nolinebreak[1] $1$.

For $2$, as before one has $n= p_1^{k_1} \dots p_m^{k_m}$ and moreover to simplify the notation suppose here that  $p_1^{k_1} = 2^k$. Suppose that $S_{p_j}$ does not contain a subgroup $\gen{ \alpha_{p_j}}$ for all prime $p_j$ dividing $n$ and prove $G$ has no non-trivial expansive subgroup with trivial $G$-core.
Let $A$ be a non-trivial subgroup of $G$ with trivial $G$-core. Then $A \cap C_n = 1$ otherwise $A \cap C_n$ would be included in the $G$-core of $A$. So $\pi$ induces an isomorphism between $A$ and a subgroup $H$ of $S$.

The subgroup $H$ is included in $\prod_{i=1}^m H_i$ where $H_i$ is the $i^{\text{th}}$-projection of $H$ on $\Aut(C_n)$. As $A$ is not trivial, so is $H$ and therefore there exists an integer $j$ such that $H_j$ is not trivial. Now one looks at the expansivity condition. Let $e$ be an element in $C_n = \prod_{i=1}^m C_{p_i^{k_i}}$ but not in $C_{C_n}(H)$ and write $e$ as $\prod_{i=1}^m e_{p_i}$ where $e_{p_i} \in C_{p_i^{k_i}}$. More precisely, if $p_j$ is odd, then take $e_{p_i}$ in $C_{C_n}(H)$ and only $e_{p_j}$ not in $C_{C_n}(H)$.  If $p_j = 2$ take the element $e_2$ to be the generator $g$ of $C_{2^k}$ and again $e_{p_i}$ in $C_{C_n}(H)$ if $p_i$ is not equal to $p_j$. Note that if $p_j = 2$, then $k > 2$ as $k=2$ is excluded by assumption and $k=1$ forces $H_2$ to be trivial. As $N_{C_n}(A) =C_{C_n}(H)$ by Lemma \ref{lemngcd}, the element $e$ is not in $N_G(A)$. Let $b$ be an element of $A$. Then $\pg{e}{b} = e \pg{b}{e^{-1}} b$ is an element of $N_G(A)$ if and only if $e \pg{b}{e^{-1}}$ is an element of $N_G(A) \cap C_n = C_{C_n} (H)$, by Lemma \ref{lemngcd}. Our purpose is to show that $e \pg{b}{e^{-1}} = 1$ and therefore $\pg{e}{A} \cap N_G(A) \leq A$. Write $b = \prod_i b_i$ such that $\pi(b_i)$ belongs to $H_i$. This is possible because $\pi$ induces an isomorphism between $A$ and $H$. Then, using Lemma \ref{cproj}, one has that $e \pg{b}{e^{-1}} = \prod_i e_{p_i} \pg{b_i}{e_{p_i}^{-1}} \in C_{C_n} (H) =\prod_i C_{C_{p_i}^{k_i}}(H_i)$, which means that for all $i$ one has $e_{p_i}\pg{b_i}{e_{p_i}^{-1}} \in  C_{C_{p_i}}(H_i)$. But by definition, for all $i$, one has $e_{p_i} \pg{b_i}{e_{p_i}^{-1}}  \in [C_{p_i^{k_i} },H_i]$ and therefore $e_{p_i} \pg{b_i}{e_{p_i}^{-1}}  \in [C_{p^{k_i} },H_i] \cap C_{C_{p_i}^{k_i}}(H_i) $. If $p_i$ is different from $p_j$ then  $\pg{b_i}{e_{p_i}}  = e_{p_i}$, as the element $e_{p_i}$ belongs to $C_{C_n}(H)$. By Theorem $2.3$ of \cite{Gor} on page 177, $[C_{p_j^{k_j} },H_j] \cap C_{C_{p_j}}(H_j)$ is trivial if $p_j$ is different from $2$. The reason is that $H_{j}$ has order prime to $p_j$ because $S_{p_j}$ does not contain a subgroup $\gen{ \alpha_{p_j}}$. For the case $p=2$ one uses Lemma \ref{p2} for $k>2$. Indeed, in this case one has shown that if $g \pg{b_2}{g^{-1}}$ belongs to $C_{C_{2^k}}(C_2)$ then $b_2 = 1$. 
 To sum up one has shown that $\pg{b_i}{e_{p_i}}  = e_{p_i}$ for all $i$. This implies that if $\pg{e}{b}$ belongs to $N_G(A)$ then $\pg{e}{b} = e \pg{b}{e^{-1}} b = b$. Finally one concludes that $\pg{e}{A} \cap N_G(A) \leq A$ and therefore $A$ is not expansive. As $A$ was an arbitrary non-trivial subgroup of $G$ with trivial $G$-core, this concludes the proof.

The third fact is a general result, see Corollary 3.7 of \cite{MoT}. One has just to prove the existence of a faithful simple module. Let $L$ be $\Ind_{C_n}^G (\xi)$ where $\xi$ is a primitive $n$th root of unity. This module is irreducible as the conjugate representations of $\xi$ by the action of $G/C_n$ are not isomorphic as it acts by automorphism on $C_n$. Moreover, the condition of primitivity on the root ensures the faithfulness of the induced module so $L$ is the wanted module.
\end{proof}
\begin{rem}{\text{}}
To ensure the non-existence of expansive subgroups with trivial $G$-core in such groups one needs, not only that $S$ does not contain a subgroup $\gen{ \alpha_p}$ for a prime $p$ dividing $n$ but also that $S_{p_j}$ does not contain a subgroup $\gen{ \alpha_{p_j}}$ for the primes $p_j$ dividing $n$. Which means that we do not want a diagonal subgroups of $\prod_{i=1}^m \Aut(C_{p_i^{k_i}})$ in $S$, see Lemma $1.1$ of \cite{The} for the terminology. However note that, depending on $G$, for example if there is no prime $p_i$ dividing $p_j -1$, the two conditions can be equivalent.
\end{rem}
\section{$\boldsymbol{(E \circ C_{p^i})\rtimes\SL(P/Z)} $}
Finally one looks at expansive subgroups in one particular example of a group with extraspecial groups in $O_p(G)$. Let $p$ be an odd prime and $E$ denotes an extraspecial group of order $p^3$ and exponent $p$. Let $P :=(E \circ C_{p^i} )$ be the central product of $E$ and a cyclic group $C_{p^i}$ over $Z(E)$ for $i \geq 1$ and $G$ be $P \rtimes \SL(P/Z)$. Then one has $Z(G) = Z(P) = C_{p^i}$ and $P/Z(P)$ is elementary abelian of rank $2$. In this section $Z$ refer to the center $Z(P)$ of $P$. First one has to understand the action of $\SL(P/Z)$ on $P$. One gives here a concrete definition of this action in term of generators. Let 
$$\begin{pmatrix}
1 &1 \\
0&1
\end{pmatrix} \quad \text{and} \quad \begin{pmatrix}
1 & 0 \\
1 &1
\end{pmatrix}$$ be two generators of $\SL(P/Z)$ acting on the vector space $P/Z(P)$ with basis $\{ \bar{f_1}, \bar{f_2} \}$. Let $f_1$ and $f_2$ be two representatives in $P$ and $z$ a generator of $Z$
The action is defined as follows on $P$

$$\begin{pmatrix}
1 &1 \\
0&1
\end{pmatrix} : \quad f_1 \mapsto f_1, \quad f_2 \mapsto f_1 f_2, \quad z \mapsto z$$

$$\begin{pmatrix}
1 &0 \\
1&1
\end{pmatrix} :\quad f_1 \mapsto f_1f_2,\quad f_2 \mapsto f_2,\quad z \mapsto z.$$
\begin{lem} \label{1.1}
Let $s$ be an element of $\SL(P/Z)$ and $e$ an element of $P$. If $s$ acts trivially on $\bar{e}$ in $P/Z$ then $\pg{s}{e} =e$.
\end{lem}
\begin{proof}
Suppose $\bar{e}$ is not in $\gen{\bar{f_1}}$ or $\gen{\bar{f_2}}$ otherwise the result is straightforward from the definition of the action of $\SL(P/Z)$ on $P$. By assumption there exists an element $h$ in $\SL(P/Z)$ such that $\pg{h}{s}$ is the matrix
$$ \begin{pmatrix}
1&\lambda \\
0&1
\end{pmatrix} \text{ for some } \lambda \in \IF_p.$$
The element $h$ corresponds to the change of basis from $(\bar{f_1}, \bar{f_2})$ to $(\bar{e}, \bar{f_2})$. By the definition of the action of $\SL(P/Z)$ on $P$ we have $\pg{\pg{h}{s}}{f_1} = f_1$ but the first term is $ hsh^{-1} f_1 h s^{-1} h^{-1}$ and therefore $\pg{s}{(h^{-1} f_1 h)} = h^{-1} f_1 h$. Using the fact that $\pg{h}{\bar{e}} = \bar{f_1}$ and so $\pg{h}{e} = zf_1$ for a $z$ in $Z$  we have 
$$\pg{s}{e} = \pg{s}{(zh^{-1} f_1 h)} = z \pg{s}{(h^{-1} f_1 h)} = z h^{-1} f_1 h = e.$$
\end{proof}

\begin{lem} \label{lemngp}
Let $S$ be a subgroup of $\SL(P/Z)$, then one has 
$$N_G(S) = N_P(S) \rtimes N_{\SL(P/Z)} (S).$$
\end{lem}
\begin{proof}
Let $x = ym$ be an element of $N_G(S)$, with $y \in P$ and $m \in \SL(P/Z)$. Let $s$ be an element of $S$. Then the fact that $\pg{x}{s} = y \pg{(\pg{m}{s})}{y^{-1}} \pg{m}{s}$ belongs to $S$ implies that $\pg{m}{s} \in S$ and therefore $m \in N_{\SL(P/Z)} (S) \leq N_G(S)$. So $xm^{-1} = y$ is an element of $N_G(S) \cap P = N_P(S)$, a the product of two elements of $N_G(S)$. The other inclusion is straightforward. 
\end{proof}
\begin{lem} \label{lemnps}
Let $S$ be a subgroup of $G$, then one has 
$$N_P(S) = N_E(S) Z.$$
\end{lem}
\begin{proof}
This is just a straightforward verification 
\begin{eqnarray*}
N_P(S) &=& \{  ez \in P \st e \in E, z \in Z \text { and } \pg{ez}{S} = S\} \\
&=& \{  ez \in P \st \pg{e}{S} = S\} \\
&=& \{ ez\in P  \st e \in N_E(S) \}  \\
&=&  N_E(S) Z.
\end{eqnarray*}
\end{proof}
\begin{lem} \label{lemat}
Let $H$ be a subgroup of $\SL(P/Z)$,  $e$ in $P$ and $G := P \rtimes\SL(P/Z)$. Then $H$ acts trivially on $e$ if and only if $e$ belongs to $N_G(H)$.
\end{lem}
\begin{proof}
Let $h$ be an element of $H$. We have $e h e^{-1} = e \pg{h}{e^{-1} } h$ so by the uniqueness of the decomposition in elements of $P$ and $\SL(P/Z)$ the element $e h e^{-1}$ belongs to $H$ if and only if  $ e \pg{h}{e^{-1} } =1$ which means that $h$ acts trivially on $e$. This must be satisfied for all $h$ in $H$ and the result follows.
\end{proof}
\begin{lem} \label{lemngs}
Let $G := P \rtimes\SL(P/Z)$, $x:=em$ with $e$ in $P$ and $m$ in $\SL(P/Z)$ and $S$ be a subgroup of the following form:
$$ \big\{ \varphi(h) h \st h \in H \big \} \text{ where } \varphi : H \rightarrow P$$
$\text{with } \varphi(hk) = \varphi(h) \pg{h}{\varphi(k)} \text{ for all } h,k \in H \text{ and } H\leq \SL(P/Z) .$ Then $x$ belongs to $N_G(S)$ if and only if
$m \in N_{\SL(P/Z)} (H)$ and $e\pg{m}{\varphi(h)} \pg{\pg{m}{h}}{e^{-1}} = \varphi(\pg{m}{h})$ for all $h \in H$.
\end{lem}
\begin{proof}
It's a straightforward calculation.
Let $s$ be an element of $S$, $s = \varphi(h) h$ for some $h$ in $H$. We have
$$\pg{em}{s} = \pg{em}{\big(\varphi(h) h\big)} =  \pg{em}{\varphi(h)} \pg{em}{h} = \pg{em}{\varphi(h)}  e \pg{\pg{m}{h}}{e^{-1}}\pg{m}{h} = e\pg{m}{\varphi(h)} \pg{\pg{m}{h}}{e^{-1}} \pg{m}{h}.$$
If $em \in N_G(S)$ then $\pg{em}{s} = \varphi(h_2) h_2$ for some $h_2$ in $H$. The unique decomposition in elements of $P$ and $\SL(P/Z)$ implies that 
$e\pg{m}{\varphi(h)} \pg{\pg{m}{h}}{e^{-1}} =  \varphi(h_2)$ and $\pg{m}{h} = h_2$. This holds for all $h \in H$.

Conversely, if $m \in N_{\SL(P/Z)} (H)$ and $e\pg{m}{\varphi(h)} \pg{\pg{m}{h}}{e^{-1}} = \varphi(\pg{m}{h})$ for all $h \in H$ then $\pg{em}{s}  = \varphi(\pg{m}{h}) \pg{m}{h} \in S$ and this holds for all $s \in S$. 
\end{proof}
\begin{lem} \label{cebceh}
Let $G := P \rtimes\SL(P/Z)$ and $S$ be a subgroup of the following form:
$$ \big\{ \varphi(h) h \st h \in H \big \} \text{ where } \varphi : H \rightarrow P$$
$\text{with } \varphi(hk) = \varphi(h) \pg{h}{\varphi(k)} \text{ for all } h,k \in H \text{ and } H\leq \SL(P/Z) .$ Then 
$$N_P(S) = C_P(S) \leq C_P(H) = N_P(H).$$
\end{lem}
\begin{proof}
First one proves that $N_P(S) = C_P(S)$. Let $e$ be an element of $N_P(S)$. As above one has, for all $h \in H$,
$$\pg{e}{s} = \pg{e}{\big(\varphi(h) h\big)} =  \pg{e}{\varphi(h)} \pg{e}{h} = \pg{e}{\varphi(h)}  e \pg{h}{e^{-1}}{h} = e\varphi(h) \pg{h}{e^{-1}} h.$$
The unique decomposition in elements of $P$ and $\SL(P/Z)$ implies that 
$$e\varphi(h) \pg{h}{e^{-1}} =  \varphi(h).$$ 
This equality implies that $h(\bar{e}) = \bar{e}$ in $P/Z$, because $P/Z$ is abelian, and so by Lemma \ref{1.1} $e=\pg{h}{e}$, for all $h \in H$, and thus $e$ belongs to $C_P(H)$. Using this in the equality above one has $\pg{e}{\varphi(h)} = \varphi (h)$. This shows that $N_P(S) \leq C_P(S)$. The other inclusion is trivial and thus $N_P(S) = C_P(S)$. The same argument with $\varphi = 1$ shows that $C_P(H) = N_P(H)$.

Finally, one proves that $C_P(S)$ is a subgroup of $C_P(H)$. Indeed let $e$ be an element of $C_P(S) = N_P(S)$. By the argument above $e$ belongs to $C_P(H)$ and the result follows.
\end{proof}
\begin{lem} \label{lemngtsp}
Let $G := P \rtimes \SL(P/Z)$ and $T$ be the non-central cyclic $p$-subgroup of $E$ generates by $f_1$. Then we have 
$$N_G(T) = N_P(T) \rtimes N_{\SL(P/Z)} (T) = \big(T \times Z \big) \rtimes \big\{ \begin{pmatrix}
\lambda&\alpha \\
0&\lambda^{-1}
\end{pmatrix}
\st \lambda \in \IF_p^{*} \text{ and }  \alpha \in \IF_p \big\}.$$
Moreover the $p$-Sylow subgroup of $N_G(T)$ is normal in $N_G(T)$.
\end{lem}
\begin{proof}
Obviously $N_P(T) \rtimes N_{\SL(P/Z)} (T) \leq N_G(T)$ and if $x := em \in N_G(T)$ with $e \in P$ and $m \in \SL(P/Z)$ then $em \in N_G( Z \times T)$ so $m \in e^{-1} N_G( Z\times T) = N_G( Z \times T)$ which means that $m$ belongs to  $\{ \begin{pmatrix}
\lambda&\alpha \\
0&\lambda^{-1}
\end{pmatrix}
\st \lambda \in \IF_p^{*} \text{ and }  \alpha \in \IF_p \}$ which is exactly $N_{\SL(P/Z)} (T) $. Therefore we can conclude with the fact that $e = x m^{-1} \in N_G(T) \cap P = N_P(T)$.
\end{proof}
\begin{lem} \label{lemh1}
Let $p$ be an odd prime number, then 
$$H^1\Big(\SL\big(E/Z(E)\big),E/Z(E)\Big) = 1.$$
\end{lem}
\begin{proof}
Recall that the group $\SL(E/Z(E))$ has $p+1$ simple $\IF_p\SL(E/Z(E))$-modules, denoted by $V_1, \dots, V_p$, where $V_i$ has dimension $i$, see \cite{Alp} page $15$ for more details.
The action of $\SL(E/Z(E))$ on $E/Z(E)$ has no trivial submodule and so $V_2 = E/Z(E)$. Let $P_1$ be the indecomposable  projective cover of $V_1$ and $\pi : P_1 \rightarrow V_1$ the corresponding homomorphism. As $p$ is odd, one can show that $P_1$ is uniserial, with three composition factors which occur as $V_1, V_{p-2}$ and $V_1$, see \cite{Alp} page $48$ for more details. 
Recall that $H^1\big(\SL(E/Z(E)),E/Z(E)\big) = \Ext_{\IF_p\SL(E/Z(E))}^1(\IF_p,E/Z(E)).$
The latter is trivial if and only if any exact sequence
\begin{figure}[h]
$$\xymatrix@1{0 \ar[r] & E/Z(E) \ar[r] &W \ar[r] & \IF_p \ar[r] &0} $$
\end{figure}\\
splits where $W$ is an $\IF_p\SL(E/Z(E))$-module. Suppose that such an exact sequence does not split for some $W$ and call $g$ the homomorphism from $W$ to $V_1 = \IF_p$. As $0\subset V_2\subset W$ with $W/V_2$ simple we conclude by Jordan-H\"older Theorem on composition series that $V_2$ is the unique non-zero proper submodule of $W$. Indeed, the only other possibility for the composition series would be $0\subset V_1\subset W$, but then the previous exact sequence would be split, which is excluded by assumption. The module $P_1$ being projective there exists a homomorphism $f : P_1 \rightarrow W$ such that $gf =\pi$. By construction $g(\im(f)) = V_1$ so $\im(f)$ is not contained in $\Ker(g) = V_2$. Therefore $\im(f) = W$ because $V_2$ is the unique non-zero proper  submodule of $W$. But this means that $W$ is isomorphic to a quotient of $P_1$ and thus $V_2$ occurs as a composition factor of $P_1$, which is a contradiction. This shows that every exact sequence above splits and $H^1\big(\SL(E/Z(E)),E/Z(E)\big) = 1$.
\end{proof}
\begin{rem} \label{remh1}
Let $G$ be a group and $A$ a $\IZ[G]$-module. One knows that $$H^1(G,A) = \cC / \cP\,$$
where 
$$\cC := \{ f : G \rightarrow A \tq f(gh) = f(g)  \pg{g}{f(h)} \quad \forall g, h \in G \}$$
and $$\cP = \{ f : G \rightarrow A \tq  \text{ there exists }a \in A \text{ with }f(g) = a^{-1}\pg{g}{a}\}.$$

The elements of $\cP$ are called the \motdef{principal crossed homomorphisms} or \motdef{$1$-coboundaries} and the elements of $\cC$ the \motdef{crossed homomorphisms} or \motdef{$1$-cocyles}\index{cocyles}.
\end{rem}
\begin{lem} \label{lemhactstriv}
Let $S$ be a subgroup of $ P \rtimes\SL(P/Z)$ such that $S \cap P = 1$ and $N_{P}(S) = Z$. Let $g$ be an element of $E$. Then, elements of $\pg{g}{S} \cap N_{G} (S)$ are of the form $\pg{g}{(\varphi(h) h)}$ where $h$ acts trivially on $g$ and $h$ is an element of $\SL(E/Z)$.    
\end{lem}
\begin{proof}
Let $s = \varphi(h) h$ be an element of $S$, and $g$ an element of $E$. One needs to know when $\pg{g}{s}$ belongs to $N_{G}(S).$ Using the following calculation,
$$ \pg{g}{s} = g \varphi(h) h g^{-1} = g \varphi(h) \pg{h}{g^{-1}} h = z g  \pg{h}{g^{-1}}  \varphi(h) h \quad \text{for some } z \in Z(P), $$
where the last equality holds because $[P,P] = Z(P)$, one remarks that $\pg{g}{s} \nolinebreak \in \nolinebreak N_{G}(S)$ if and only if $g  \pg{h}{g^{-1}} \in N_{G}(S)$ because $z$ and $ \varphi(h) h$ belong to $N_G(S)$. This holds only if $g  \pg{h}{g^{-1}} \in Z$ as $g  \pg{h}{g^{-1}} \in P$ and $N_G(S) \cap P = Z$ by assumption. But this implies that $h(\overline{g}) = \overline{g}$ in $P/Z(P)$. Therefore $h$ acts trivially on $g$ by Lemma \ref{1.1}.
\end{proof}
\begin{thm} \label{psl2exp} \label{thermk}
Let $G := P \rtimes\SL(P/Z)$, then $G$ has no non-trivial expansive subgroup with trivial $G$-core. 
\end{thm}
\begin{proof}
Let $S$ be a non-trivial expansive subgroup of $G$ with trivial $G$-core. We must have $S \cap Z = 1$, otherwise $Z$ would be contained in the $G$-core of $S$. So only two cases are possible:
\begin{enumerate}
\item $S \cap P =1$ or 
\item $S \cap P =T$ for $T$ a non-central $p$-subgroup of $E$ of order $p$.
\end{enumerate}
In the first case, one can check that $S$ is a subgroup of the following form :
$$ \big\{ \varphi(h) h \st h \in H \big \} \text{ where } \varphi : H \rightarrow P,$$
for $H\leq \SL(P/Z)$ and $\text{with } \varphi(hk) = \varphi(h) \pg{h}{\varphi(k)} \text{ for all } h,k \in H$, i.e. $\varphi$ is a $1$-cocycle.

Assume first that $\varphi =1$ and so $S=H$. By Lemma \ref{lemnps}, $N_P(S) = N_E(S) Z$, 
so only two cases are possible $N_P(S) = Z$ or $N_P(S) = Z \times Q$ for $Q$ a non-central $p$-subgroup of $E$. Indeed, by the structure of subgroups of $E$, the subgroup $N_E(S)$ could only be $Z(E)$ or $Z(E) \times Q$.  To start, suppose that  $Z= N_P(S)$. Let $g$ be an element of $E$ but not in $N_G(S)$. So we can write $\pg{g}{S}$ as $\big\{ g \pg{s}{g^{-1}} s \st s \in S \big \}$ and using Lemma \ref{lemngp} we have
$$ \pg{g}{S} \cap N_G(S) =  \big\{ g \pg{s}{g^{-1}} s \st  g \pg{s}{g^{-1}} \in Z  \text{ and } s \in S \big \}. $$
If $g \pg{s}{g^{-1}} \in Z$ then $s$ acts trivially on $\overline{g}$ in $P/Z$, which means that $s$ acts trivially on $g$ by Lemma \ref{1.1}. So the fact that $g \pg{s}{g^{-1}} $ belongs to $Z$ implies that $g \pg{s}{g^{-1}} = 1$ as well as $\pg{g}{s} = s$ and thus $\pg{g}{S} \cap N_G(S) \leq S$. This shows that $S$ is not expansive.

Now suppose that $N_P(S) = Z \times Q$ for $Q$ a non-central $p$-subgroup of $E$, which implies that $S$ acts trivially on $Q$ and so, up to conjugation, we have
$$S \leq \Big \{ \begin{pmatrix}
1 &\alpha \\
0&1
\end{pmatrix}
\st  \alpha \in \IF_p  \Big \}.$$

Moreover, $S=1$ is excluded by assumption, so we have equality. Let $s$ be an element of $S$ and $g$ be 
$$\begin{pmatrix}
0&1 \\
-1&0
\end{pmatrix}.$$
Then $$\pg{g}{s} = \begin{pmatrix}
0&1 \\
-1&0
\end{pmatrix}
 \begin{pmatrix}
1&\alpha_s \\
0&1
\end{pmatrix}
\begin{pmatrix}
0&-1 \\
1&0
\end{pmatrix} = \begin{pmatrix}
1 &0 \\
-\alpha_s&1
\end{pmatrix}$$
which is an element of $N_{\SL(P/Z)}(S)$ only if $\alpha_s = 0$. Therefore one has $\pg{g}{s}= 1 = s$ and $\pg{g}{S} \cap  N_{\SL(P/Z)}(S) \leq S$.

This shows that $S$ is not expansive if $\varphi =1$.

Assume now that $\varphi \neq 1$. Again, by Lemma \ref{lemnps}, one has only two possibilities for $N_P(S)$, either $N_P(S) = QZ$ or $N_P(S) = Z$. In the first case, by Lemmas \ref{lemat} and \ref{cebceh}, $H$ acts trivially on $Q$ and therefore, up to conjugation, $H$ is $$\Big \{ \begin{pmatrix}
1 &\alpha \\
0&1
\end{pmatrix}
\st  \alpha \in \IF_p  \Big \}.$$
Let $g = \begin{pmatrix}
0&1 \\
-1&0
\end{pmatrix}$ and $s = \varphi(h) h$ an element of $S$.
If $\pg{g}{s} = \pg{g}{\varphi(h)}\pg{g}{h}$ belongs to $\pg{g}{S} \cap  N_{G}(S)$ then $\pg{g}{h} \in N_{\SL(P/Z)}(H)$ by Lemma \ref{lemngs}. Write $h$ as $\begin{pmatrix}
1&\alpha \\
0&1
\end{pmatrix}$  then 
$$\pg{g}{h} = \begin{pmatrix}
0&1 \\
-1&0
\end{pmatrix}
 \begin{pmatrix}
1&\alpha \\
0&1
\end{pmatrix}
\begin{pmatrix}
0&-1 \\
1&0
\end{pmatrix} = \begin{pmatrix}
1 &0 \\
-\alpha&1
\end{pmatrix}$$
which is an element of $N_{\SL(P/Z)}(H)$ only if $\alpha = 0$ and so $\pg{g}{h}= h = 1$. Thus $s = 1$ and therefore we conclude $\pg{g}{S} \cap  N_{G}(S) =1$ and $\pg{g}{s} = s$. 

Next we suppose that $N_P(S)$ is equal to $Z$. Recall one has fixed a basis $\gen{\bar{f_1}, \bar{f_2}}$ of $E/Z(E)$. 
Let $s = \varphi(h) h$ be an element of $S$,  and $g = f_1$ a representative of $\bar{f_1}$ in $E$. One needs to know when $\pg{g}{s}$ belongs to $N_{G}(S).$ By Lemma \ref{lemhactstriv}, one knows that
 $h(\overline{g}) = \overline{g}$ in $P/Z$. The same argument works for $g=f_2$. Therefore, using Lemma \ref{1.1}, one has
$$ \pg{f_1}{S} \cap N_G(S) \leq \{ \pg{f_1}{(\varphi(h)} h) \tq h \in H \text{ and } \pg{h}{f_1} = f_1  \}  \text{ and }$$
 $$\pg{f_2}{S} \cap N_G(S) \leq \{ \pg{f_2}{(\varphi(k)} k) \tq k \in H \text{ and } \pg{k}{f_2} = f_2 \}.$$
These sets are isomorphic to a unipotent group of order $p$ as they act trivially on $f_1$ respectively $f_2$. Moreover, either one of these intersections is trivial and then one takes respectively $g$ to be $f_1$ or $f_2$, or both intersections are not trivial. In the latter case, $H = \SL(P/Z)$ since $H$ contains two different transvections, the one which acts trivially on $f_1$ and the one which acts trivially on $f_2$. By Lemma \ref{lemh1} , $H^1\big(\SL(P/Z),P/Z \big) = 1$ and so there exists $\overline{a} \in P/Z$ such that $\overline{\varphi(h)} = \overline{a}^{-1} h(\overline{a}) $ for all $h \in \SL(P/Z)$, see Remark \ref{remh1}. The element $\bar{a}$ is not trivial otherwise it is the case where $\varphi =1$ that has been treated before. Thus $a$ does not belongs to $Z = N_E(S)$. Then, for a fixed $h$ in $\SL(P/Z)$ there exists an element $z_h \in Z$ such that $\varphi(h) = z_h a^{-1} \pg{h}{a}$. Let $g$ be equal to $a$ which is as mentioned not an element of $N_G(S)$. One looks at $\pg{a}{S} \cap N_G(S)$. Let $s$ be an element of $S$, then $s =  z_h a^{-1} \pg{h}{a} h$ and the fact that $\pg{a}{s}$ has to belong to $N_G(S)$ implies that $h$ acts trivially on $a$. This is the same reasoning as above for $\pg{f_1}{s}$. So $s = z_h h$ and $\pg{a}{s} = z_h h = s \in S$, which proves that $\pg{a}{S} \cap N_G(S) \leq S.$

For the second case, namely $S \cap P = T$ for $T$ a non-central $p$-subgroup of $E$ of order $p$, remark that by Lemma \ref{lemngtsp}, one has up to conjugation
$$N_G(S) \leq N_G(T) = ( T \times Z ) \rtimes  \{ \begin{pmatrix}
\lambda&\alpha \\
0&\lambda^{-1}
\end{pmatrix}
\st \lambda \in \IF_p^{*} \text{ and }  \alpha \in \IF_p  \},$$ 
where the basis is chosen with the first element in $T$.
Let $s$ be an element of $S$ and $g = \begin{pmatrix}
0&1 \\
-1&0
\end{pmatrix}$. One can write $s$ as $tz \begin{pmatrix}
\lambda_s &\alpha_s \\
0&\lambda_s^{-1}
\end{pmatrix}$ for some $t$ in $T$ and $z$ in $Z$. Then, the fact that the element $\pg{g}{s}$ belongs to $N_G(S)$ implies that $\alpha_s = 0$ and $t = 1$. Indeed, $\pg{g}{t}z \pg{g}{\begin{pmatrix}
\lambda_s&\alpha_s \\
0&\lambda^{-1}_s
\end{pmatrix}}$ belongs to $N_G(S)$ only if $\pg{g}{t}z$ belongs to $ T \times Z$ and $\pg{g}{\begin{pmatrix}
\lambda_s&\alpha_s \\
0&\lambda^{-1}_s
\end{pmatrix}}$ to $\{ \begin{pmatrix}
\lambda&\alpha \\
0&\lambda^{-1}
\end{pmatrix}
\st \lambda \in \IF_p^{*} \text{ and }  \alpha \in \IF_p  \}$. With the following calculation

$$ \begin{pmatrix}
0&1 \\
-1&0
\end{pmatrix}
 \begin{pmatrix}
\lambda_s &\alpha_s \\
0&\lambda^{-1}_s
\end{pmatrix}
\begin{pmatrix}
0&-1 \\
1&0
\end{pmatrix} = \begin{pmatrix}
\lambda^{-1}_s &0 \\
-\alpha_s&\lambda_s
\end{pmatrix}$$
one sees that the latter occurs only if $\alpha_s = 0$. Moreover $\pg{g}{t}z$ belongs to $ T \times Z$ only if $t=1$ as $g$ sends $t$ to $T^c$. 

Therefore $s = zm := z \begin{pmatrix}
\lambda_s& 0 \\
0&\lambda^{-1}_s
\end{pmatrix}$ and so, because $o(z)$ divides $p^i$ and $\lambda_s^p = \lambda_s$, the element $s^{o(z)} =\begin{pmatrix}
\lambda_s& 0 \\
0&\lambda_s^{-1}
\end{pmatrix}$ belongs to $S$. This implies that $sm^{-1} = z$ belongs to $S$ as a product of elements of $S$. Therefore $z = 1$ because one has $S \cap P = T$. Finally, we have $s = \begin{pmatrix}
\lambda_s& 0 \\
0&\lambda_s^{-1}
\end{pmatrix}$ and $\pg{g}{s} = \begin{pmatrix}
\lambda_s^{-1} & 0 \\
0&\lambda_s
\end{pmatrix} = s^{-1}$ which is an element of $S$ and so  $\pg{g}{S} \cap N_G(S) \leq S.$
\end{proof}
\begin{rem}
In the case $\varphi \neq 1$ and $N_P(S) = Z$, one has to notice that one used $\pg{f_1}{S} \cap N_G(S)$ and $\pg{f_2}{S} \cap N_G(S)$ to obtain information on $S$. But at the end one proves the non-expansivity of $S$ by looking at $\pg{a}{S} \cap N_G(S)$. 
\end{rem}
\begin{thm} \label{pkexp}
Let $G := P \rtimes K$, with $K$ a subgroup of $\SL(P/Z)$. Then $G$ has no non-trivial expansive subgroup with trivial $G$-core if and only if $K$ is not contained in a Borel subgroup of $\SL(P/Z)$.
\end{thm}
\begin{proof}
Suppose first that $K$ is contained in a Borel subgroup of $\SL(P/Z)$. Let $T$ be the $p$-subgroup of $E$ of order $p$ normalized by the Borel subgroup. Then, the normalizer $N_G(T)$ is $Z \times T \rtimes K$ and for all $g$ in $G$ but not in $N_G(T)$ we have
$$ (N_G(T) \cap \pg{g}{T})T = \pg{g}{T}T = Z \times T$$ 
because  $\pg{g}{T}$ is contained in $ Z \times T$ but not equal to $T$. Then the $N_G(T)$-core of $Z \times T$ is $Z \times T$ and so $T$ is an expansive subgroup with trivial $G$-core.

Conversely, suppose that $K$ is not contained in a Borel subgroup of $\SL(P/Z)$. If $p$ divides $\abs{K}$ then $K = \SL(P/Z)$. Indeed, the number of $p$-Sylow subgroups of $K$ is either $1$ or $p+1$. In the first case, the $p$-Sylow subgroup, denoted by $U$, is normal. We would have, up to conjugation, that 
$$U = \Big \{ \begin{pmatrix}
1 &\alpha \\
0&1
\end{pmatrix}
\st  \alpha \in \IF_p  \Big \} \leq K \leq  N_{\SL(P/Z)}(U).$$
So $K$ would be contained in a Borel subgroup which is impossible by assumption. Moreover, if the number of $p$-Sylow subgroups is $p+1$, then $K$ contains all the transvections which generate $\SL(P/Z)$. The case $K= \SL(P/Z)$ has already been treated therefore we can assume that $p$ doesn't divide the order of $K$. 

Let $S$ be a non-trivial subgroup with trivial $G$-core. We must have $S \cap Z(P) = 1$, otherwise $Z(P)$ would be contained in the $G$-core of $S$. So only two cases are possible:
\begin{enumerate}
\item $S \cap P =1$ or 
\item $S \cap P =T$ for $T$ a non-central $p$-subgroup of $E$ of order $p$.
\end{enumerate}
We start with $S \cap P =1$. As $p$ does not divide the order of $K$ then $H^1(S,E) \nolinebreak = \nolinebreak 1$ and so up to conjugation $S$ is a subgroup of $K$. Obviously, we know that $Z(P) \leq N_P(S) \lneqq P$. By Lemma \ref{lemnps}, $N_P(S) = N_E(S) Z(P)$, so only two cases are possible $N_P(S) = Z(P)$ or $N_P(S) = Z(P) \times Q$ for $Q$ a non-central $p$-subgroup of $E$. Indeed, by the structure of subgroups of $E$, the subgroup $N_E(S)$ could only be $Z(E)$ or $Z(E) \times Q$.  To start suppose that  $Z(P)= N_P(S)$. Let $g$ be an element of $E$ but not in $N_G(S)$. So we can write $\pg{g}{S}$ as $\big\{ g \pg{s}{g^{-1}} s \st s \in S \big \}$ and using Lemma \ref{lemngp} we have
$$ \pg{g}{S} \cap N_G(S) =  \big\{ g \pg{s}{g^{-1}} s \st  g \pg{s}{g^{-1}} \in Z(P)  \text{ and } s \in S \big \}. $$
If $g \pg{s}{g^{-1}} \in Z(P)$ then $s$ acts trivially on $\overline{g}$ in $P/Z(P)$, which means that $s$ acts trivially on $g$ by Lemma \ref{1.1}. So the fact that $g \pg{s}{g^{-1}} $ belongs to $Z(P)$ implies that $g \pg{s}{g^{-1}} = 1$ and thus $\pg{g}{S} \cap N_G(S) \leq S$. This shows that $S$ is not expansive.

Now suppose that $N_P(S) = Z(P) \times Q$ for $Q$ a non-central $p$-subgroup of $E$, which implies that $S$ acts trivially on $Q$ and so, up to conjugation, we have
$$S \leq \Big \{ \begin{pmatrix}
1 &\alpha \\
0&1
\end{pmatrix}
\st  \alpha \in \IF_p  \Big \}.$$
Moreover, $S=1$ is excluded by assumption, so we have equality. But then $p$ divides the order of $S$ and so of $K$, which is impossible. So this case can actually not occur. 

For the second case, namely $S \cap P = T$, one observes that
\begin{eqnarray*}
N_G(S) &\leq & N_G(T) = ( T \times Z(P) ) \rtimes C \\ &:=& ( T \times Z(P) ) \rtimes  \{ \begin{pmatrix}
\beta&\alpha \\
0&\beta^{-1}
\end{pmatrix}
\st \beta \in \IF_p^{*} \text{ and }  \alpha \in \IF_p  \}\cap K \\
&\leq & ( T \times Z(P) ) \rtimes \{ \begin{pmatrix}
\beta & (\beta^{-1} - \beta) \gamma \\
0&\beta^{-1}
\end{pmatrix}
\st \beta \in \IF_p^{*} \}.
\end{eqnarray*} 
The last inclusion holds for a fixed $\gamma$ because $p$ does not divide the order of $K$. Note that here the first vector of the basis belongs to $T$. Moreover by Schur-Zassenhaus lemma, $S$ is of the form $T \rtimes D$ where, up to conjugation by an element of $N_G(T)$, $D$ is a subgroup of $C$. As $Z$ acts trivially on $D$ and $T$ is contained in $S$, one can assume that $D$ is a subgroup of $C$. Let's prove that $S$ is not expansive. Let $s$ be an element of $S$ and $g$ an element of $K$ but not in $N_G(T\times Z)$. This element exists because $K$ is not contained in a Borel subgroup. One can write $s$ as $tm:= t \begin{pmatrix}
\lambda & (\lambda^{-1} - \lambda) \gamma \\
0&\lambda^{-1}
\end{pmatrix}$ for some $t$ in $T$ and $\lambda \in \IF_p^{*}$. Then, the fact that the element $\pg{g}{s}$ belongs to $N_G(S)$ implies that $t = 1$. Indeed, $\pg{g}{t}$ belongs to $ T \times Z(P)$ is only possible if $t=1$ as $g$ does not belong to $N_G(T \times Z)$. 

Therefore $\pg{g}{s}$ is reduced to $\pg{g}{m}$ which belongs to $C$ as it belongs to $N_G(S)$. But $m$ belongs to $C$ as well as it belongs to $D$. Finally, since $C$ is cyclic we conclude that if the element $\pg{g}{m}$ belongs to $C$ then it must actually belong to $D$, by simply looking at its order, as it is the same as the order of $m$. Thus $\pg{g}{s}$ belongs to $S$. Finally, one concludes that $\pg{g}{S} \cap N_G(S) \leq S$ and therefore $S$ is not expansive.
\end{proof}
\begin{rem}
One tried to look at more examples of groups containing an extraspecial group in the Fitting subgroup. For example, let $p$ and $q$ be two different prime numbers and $E_p$, respectively $E_q$, denotes an extraspecial group of order $p^3$, respectively $q^3$ and exponent $p$, respectively $q$. Let $G_1$ be $E_p \rtimes \SL(E_p/Z_p)$ and $G_2$ be $E_q \rtimes \SL(E_q/Z_q)$. Finally let $G$ the direct product of $G_1$ and $G_2$. Then one proves that $G$ has no non-trivial expansive subgroup $S$ with $S \cap (E_p \times E_q) = 1$.

Another example, let $E$ be an extraspecial group of order $p^{1+2n}$ and exponent $p$ for an odd prime $p$ and $G$ be $E \rtimes \Sp(E/Z)$ one proved that if $S$ is a subgroup of $\Sp(E/Z)$, then $S$ is not expansive. One has proved a similar result for subgroups $S$ such that $S \cap E = 1$ with added conditions.

However the result is not true if the extraspecial group if a $2$-group. Indeed, the group $Q_8 \rtimes S_3$ is a Roquette group. The Fitting subgroup is $Q_8$, an extraspecial $2$-group and in \cite{B-T} it is shown that $S_3$ is an expansive subgroup. 
\end{rem}

Alex Monnard, Section de math\'ematiques, EPFL,\\
Station 8, CH-1015 Lausanne, Switzerland. \\
\tt Alex.Monnard@epfl.ch


\end{document}